\documentclass[11pt,a4paper]{amsart}
\usepackage{amscd,amssymb,amsopn,amsmath,amsthm,mathrsfs,graphics,amsfonts,enumerate,verbatim,calc}
\usepackage[dvips]{graphicx}

\usepackage{color}

\usepackage{amssymb,amsmath}

\headsep=1cm \numberwithin{equation}{section}
\newtheorem{theorem}{Theorem}[section]
\newtheorem{lemma}[theorem]{Lemma}

\newtheorem{corollary}[theorem]{Corollary}
\newtheorem{claim}[theorem]{Claim}

\theoremstyle{definition}
\newtheorem{maintheorema}{Main Theorem}
\newtheorem{definition}[theorem]{Definition}
\theoremstyle{remark}

\newtheorem{acknowledgement}{Acknowledgement}

\newcommand{\Spec}{\operatorname{Spec}}

\newcommand{\rad}{\operatorname{rad}}

\newcommand{\Ht}{\operatorname{ht}}

\newcommand{\Max}{\operatorname{Max}}

\newcommand{\Reg}{\operatorname{Reg}}
\newcommand{\Sing}{\operatorname{Sing}}

\newcommand{\fm}{\frak{m}}
\newcommand{\fp}{\frak{p}}
\newcommand{\fq}{\frak{q}}
\newcommand{\fa}{\frak{a}}

\newcommand{\fn}{\frak{n}}

\newcommand{\fS}{\frak{S}}

\begin{document}
\title[Ideal-adic completion of quasi-excellent rings (after Gabber)]
{Ideal-adic completion of quasi-excellent rings (after Gabber)}
\dedicatory{Dedicated to Professor~Jun-ichi~Nishimura on the~occasion of his~70th~birthday.}

\author[K. Kurano]{Kazuhiko Kurano}
\address{Department of Mathematics, School of Science and Technology,
Meiji University, Higashimata 1-1-1, Tama-ku, Kawasaki 214-8571, Japan}
\email{kurano@meiji.ac.jp}

\author[K. Shimomoto]{Kazuma Shimomoto}
\address{Department of Mathematics, College of Humanities and Sciences, Nihon University, Setagaya-ku, Tokyo 156-8550, Japan}
\email{shimomotokazuma@gmail.com}

\thanks{2010 {\em Mathematics Subject Classification\/}: 13B35, 13F25, 13F40}

\keywords{Ideal-adic completion, lifting problem, local uniformization, excellent ring,  quasi-excellent ring}


\begin{abstract}
In this paper, we give a detailed proof to a result of Gabber (unpublished) on the lifting problem of quasi-excellent rings, extending the previous work on Nishimura-Nishimura. As a corollary, we establish that an ideal-adic completion of an excellent (resp.\ quasi-excellent) ring is excellent  (resp.\ quasi-excellent).
\end{abstract}

\maketitle

\setcounter{section}{0}

\section{Introduction}

Throughout this paper, we assume that all rings are commutative and possess an identity. The aim of this article is to give a detailed proof to the following theorem (see Theorem \ref{quasi-excellent}).

\begin{maintheorema}[Nishimura-Nishimura, Gabber]
\label{MT1}
Let $A$ be a Noetherian ring, and $I$ an ideal of $A$. Assume that $A$ is $I$-adically complete. Then, if $A/I$ is quasi-excellent, so is $A$.
\end{maintheorema}

This result was proved in characteristic 0 by Nishimura-Nishimura in \cite{Nishimura87}, using the resolution of singularities. More recently, the general case was settled by Gabber, using his theorem of weak local uniformization of quasi-excellent schemes instead of resolutions of singularities. 
The idea of his proof is sketched in a letter \cite{Laszlo07} from Gabber to Laszlo. The above theorem is a special and difficult case of the \textit{Lifting Problem}, which was formulated by Grothendieck \cite[Remarque (7.4.8)]{Gro}. For the precise definition of the lifting problem as well as its variations, we refer the reader to Definition~\ref{GrothendieckLifting} in Appendix~\ref{appA}. 
As an immediate and important corollary, we obtain the following theorem (see Appendix~\ref{appA}).

\begin{maintheorema}
\label{MT2}
Let $A$ be an excellent (resp.\ quasi-excellent) ring with an ideal $I \subset A$. Then the $I$-adic completion of $A$ is an excellent (resp.\ quasi-excellent) ring. In particular, if $A$ is excellent (resp.\ quasi-excellent), then the formal power series ring $A[[x_1,\ldots,x_n]]$ is excellent (resp.\ quasi-excellent).
\end{maintheorema}

Here is an outline of the paper.

In \S~\ref{quasi-excellent1}, we fix notation and definitions concerning excellent rings for which we refer the reader to \cite{Ma80} and \cite{Nishimura87}. 

In \S~\ref{quasi-excellent2}, we begin with the notions of quasi-excellent schemes and alteration coverings. Then we recall the recent result of Gabber on the existence of weak local uniformization of a quasi-excellent Noetherian scheme as an alteration covering.

In \S~\ref{quasi-excellent3}, using Gabber's weak local uniformization theorem, we give a proof to the classical theorem of Brodmann and Rotthaus in the full generality. This result is an important step toward the proof of the main result of this paper.

In \S~\ref{quasi-excellent4}, we finish the proof of the main result after proving a number of intermediate lemmas based on the ideas explained in \cite{Nishimura87} and \cite{Laszlo07}. Gabber's weak local uniformization theorem plays a role in this section too.

We introduce notions of the lifting property, the local lifting property, etc.
We make a table on the known results in Appendix~\ref{appA}.

\section{Notation and conventions}
\label{quasi-excellent1}

We use the following notation. Let $I$ be an ideal of a ring $A$. Then denote by $V(I)$ the set of all prime ideals of $A$ containing $I$. 
Put $D(I)=\Spec A \setminus V(I)$.
If $(A,\fm)$ is a local ring, we denote by $\widehat{A}$ or $A^\wedge$ the $\fm$-adic completion. For an integral domain $A$, denote by $Q(A)$ the field of fractions of $A$. 

We refer the reader to the definition of \textit{geometrically regular} and \textit{regular morphism} to \cite[(33.A)]{Ma80}. We refer the reader to the definition of \textit{$G$-ring} and \textit{$Z$-ring} to \cite[Definition~(0.1)]{Nishimura87}. We refer the reader to the definition of \textit{catenary}, \textit{universally catenary}, \textit{Nagata ring} (or \textit{universally Japanese ring}), \textit{$J_2$-ring}, \textit{quasi-excellent}, \textit{excellent} to (14.B),  (31.A),  (32.B),  (34.A) in \cite{Ma80}. 

Any field, the ring of integers and complete Noetherian local rings are excellent rings. Any semi-local $G$-ring is a $J_2$-ring (cf. \cite[Theorem 76]{Ma80}). Hence semi-local $G$-rings are quasi-excellent.
Quasi-excellent rings are Nagata rings (cf. \cite[Theorem 78]{Ma80}).

\section{Gabber's weak local uniformization theorem for quasi-excellent schemes}
\label{quasi-excellent2}

In this section, let us recall a recent result on the existence of weak local uniformizations for quasi-excellent Noetherian schemes, due to Gabber.

\begin{definition}
A Noetherian scheme $X$ is \textit{quasi-excellent} (resp.\ \textit{excellent}), if $X$ admits an open affine covering, each of which is the spectrum of a quasi-excellent ring (resp. an excellent ring). Once this condition holds, then any other open affine covering has the same property.
\end{definition}

We introduce the notion of alteration covering of schemes.

\begin{definition}
In this definition, we assume that all schemes are reduced and Noetherian, and all morphisms are maximally dominating, generically finite morphisms of finite type between Notherian reduced schemes.

Let $Y$ be a Noetherian integral scheme. We say that a finite family of scheme maps $\{ \phi_i :X_i \to Y \}_{i=1,\ldots, m}$ is an \textit{alteration covering} of $Y$, 
if there exist a proper surjective morphism $f:V \to Y$, a Zariski open covering $V=\bigcup_{i=1}^m V_i$, together with a family of scheme maps $\{\psi_i:V_i \to X_i\}_{i=1,\ldots,m}$ such that the following diagram commutes for each $i= 1, \ldots,m$
\begin{equation}
\label{diagram}
\begin{array}{ccccc}
V_i & {\longrightarrow} & V &\\
\hphantom{\scriptstyle \psi_i}\downarrow{\scriptstyle \psi_i} &&  \hphantom{\scriptstyle f}\downarrow{\scriptstyle f}
& & \\
X_i & \stackrel{\phi_i}{\longrightarrow} & Y\\
\end{array}
\end{equation}
where $V_i \to V$ is the natural open immersion.

If $X_i$ is a regular integral scheme for each $i = 1,\ldots,m$, 
we say that $\{ \phi_i :X_i \to Y \}_{i = 1,\ldots,m}$ is a \textit{regular alteration covering} of $Y$.
\end{definition}

We refer the reader to \cite{Book} for alteration coverings or the alteration topology in the general situation. The definition adopted in \cite{Book} looks slightly different from the above one. However, we may resort to \cite[Th\'eor\`eme 3.2.1; EXPOS\'E II]{Book}. A family of morphisms of the type we consider $\{ \phi_i :X_i \to Y \}_{i=1,\ldots, m}$, or a family of morphisms in ${\rm alt}/Y$ is a covering family for the alteration topology if and only if it is a covering family for the $h$-topology. Let us state Gabber's weak local uniformization theorem for which we refer the reader to \cite{Book}.

\begin{theorem}[Gabber]
\label{Gabberuniformization}
Assume that $Y$ is a quasi-excellent Noetherian integral scheme. Then there exists a regular alteration covering of $Y$.
\end{theorem}

Using the valuative criterion for proper maps and Gabber's theorem, we obtain the following corollary. That is the reason why the above theorem is called ``weak local uniformization theorem" for quasi-excellent schemes.

\begin{corollary}
Assume that $A$ is a quasi-excellent domain. Then there exists a finite field extension $K/Q(A)$ such that if $R$ is a valuation domain satisfying $Q(R)=K$ and $A \subset R$, then there exists a regular domain $B$ such that $A \subset B \subset R$ and $B$ is a finitely generated $A$-algebra.
\end{corollary}

\section{A generalization of a theorem of Brodmann and Rotthaus}
\label{quasi-excellent3}

The purpose of this section is to prove the following theorem. It was proved by Brodmann and Rotthaus \cite{BroRot80} for rings containing a field of characteristic zero, using the resolution of singularities by Hironaka. Let $\rad(A)$ be the Jacobson radical of a ring $A$.

\begin{theorem}
\label{BrodmannRotthaus}
Let $A$ be a Noetherian ring with an ideal $I \subset \rad(A)$. Assume that $A/I$ is quasi-excellent and $A$ is a $G$-ring. Then $A$ is a $J_2$-ring. In other words, $A$ is quasi-excellent.
\end{theorem}

We need to prove a number of lemmas before proving this theorem. Let $X$ be a Noetherian scheme. We denote by
$\Reg(X)$ the regular locus of $X$,
and put ${\rm Sing}(X) = X \setminus \Reg(X)$.

\begin{lemma}
\label{lemmaopenset}
Let $A$ be a Noetherian ring with an ideal $I \subset A$ and
let $\pi:X \to \Spec A$ be a scheme map of finite type. Assume that $A/I$ is a $J_2$-ring. Then $\Reg(X) \cap \pi^{-1}\big(V(I)\big)$ is open in $\pi^{-1}\big(V(I)\big)$.
\end{lemma}

\begin{proof}
For the proof, we may assume that $X$ is an affine scheme.
Let $B$ be an $A$-algebra of finite type and put $X=\Spec B$.
First, let us prove the following claim:

\begin{claim}
\label{subclaim}
Assume that $Z \subset V(IB)$ is a closed subset. Then $\Reg(X) \cap Z$ is constructible in $Z$.
\end{claim}

\begin{proof}[Proof of the claim]
We prove it by Noetherian induction. So let us suppose that any proper closed subset of $Z$ satisfies the conclusion of the claim. We may assume that $Z=V(\fq)$ for some prime ideal $\fq \supset IB$. Since $A/I$ is a $J_2$-ring, we have
\begin{equation}
\label{open}
\mbox{$\Reg\big(\Spec(B/\fq)\big)$ is a non-empty open set in $\Spec(B/\fq)$.}
\end{equation}

Assume that $\fq \notin \Reg(X)$. Then we have $\Reg(X) \cap Z=\emptyset$ and this is evidently constructible. Next, assume that $\fq \in \Reg(X)$. Then since $B_{\fq}$ is a regular local ring, the maximal ideal $\fq B_{\fq}$ is generated by a regular sequence. This together with $(\ref{open})$ implies the following:
\begin{enumerate}
\item[-]
There exists an element $f \in B \setminus \fq$ such that $\fq B[f^{-1}]$ is generated by a $B[f^{-1}]$-regular sequence and $B[f^{-1}]/\fq B[f^{-1}]$ is a regular ring.
\end{enumerate}

Hence if $\fp \in \Spec B$ is taken such that $\fp \in V(\fq)$ and $\fp \notin V(\fq+fB)$, then $B_{\fp}$ is regular. We have the decomposition:
$$
\Reg(X) \cap V(\fq)=\Big(\Reg(X) \cap V(\fq+fB)\Big) \cup \Big(\Reg(X) \cap \big(V(\fq) \setminus V(\fq+fB)\big)\Big)
$$
By Noetherian induction hypothesis, we see that $\Reg(X) \cap V(\fq+fB)$ is constructible. On the other hand, we have $\Reg(X) \cap \big(V(\fq) \setminus V(\fq+fB)\big)=V(\fq) \setminus V(\fq+fB)$ and this is clearly open in $V(\fq)$. We conclude that $\Reg(X) \cap V(\fq)$ is constructible, which finishes the proof of the claim.
\end{proof}

We can finish the proof of Lemma~\ref{lemmaopenset} in the following way. It is easy to show that $\Reg(X) \cap \pi^{-1}\big(V(I)\big)$ is closed under taking generalizations of points inside $\pi^{-1}\big(V(I)\big)$.
Let us recall that a subset of a Noetherian scheme is open if and only if it is constructible and stable under generalization of points \cite[II. Ex. 3.17, 3.18]{Har77}. 
 Combining Claim \ref{subclaim}, we conclude that $\Reg(X) \cap \pi^{-1}\big(V(I)\big)$ is open in $\pi^{-1}\big(V(I)\big)$.
\end{proof}

Using this lemma, we can prove the following crucial fact.

\begin{lemma}
\label{alterationcover}
Let $B$ be a Noetherian domain and let us choose $\fq \in \Spec B$ and an ideal $J \subset B$. Assume that $B_{\fq}$ is a $G$-ring and $B/J$ is a $J_2$-ring. 
Then there exists $b \in B \setminus \fq$ together with an alteration covering $\{\phi_{b,i}:X_{b,i} \to \Spec B[b^{-1}]\}_i$ such that
\begin{equation}
\label{openset}
 \phi_{b,i}^{-1}\biggl(\Spec \frac{B[b^{-1}]}{J B[b^{-1}]}\biggl) \subset 
 \Reg(X_{b,i})
\end{equation}
for all $i$.
\end{lemma}

\begin{proof}
Since $B_{\fq}$ is a quasi-excellent local domain by \cite[Theorem 76]{Ma80}, there exists a regular alteration covering $\{ \phi_i:X_i \to \Spec B_{\fq} \}_{i = 1, \ldots, m}$, together with a proper map $f: V \to \Spec B_{\fq}$ with $V:=\bigcup_{i=1}^m V_i$ and $\{\psi_i:V_i \to X_i \}_i$ by Theorem \ref{Gabberuniformization}. 
By Chow's lemma, we may assume that $f:V \to \Spec B_{\fq}$ is projective. 
Then we can find an element $\tilde{b} \in B \setminus \fq$ and an alteration covering $\{\phi_{\tilde{b},i}:X_{\tilde{b},i} \to \Spec B[\tilde{b}^{-1}]\}_{i = 1, \ldots, m}$ such that $\phi_i=\phi_{\tilde{b},i} \otimes_{B[\tilde{b}^{-1}]} B_{\fq}$. 

Remark that, if $\phi_{\tilde{b},i}(s)$ is a generalization of $\fq B[\tilde{b}^{-1}]$, then we have $s \in X_i$. In particular, 
\begin{equation}\label{reg}
\mbox{$\mathcal{O}_{X_{\tilde{b},i},s}$ is regular. }
\end{equation}

Let us put
$$
\mbox{$Z_{\tilde{b}}:=\Spec \biggl(\dfrac{B[\tilde{b}^{-1}]}{J B[\tilde{b}^{-1}]}\biggl)$, which is a closed subset of $\Spec B[\tilde{b}^{-1}]$.}
$$
Then we find that
$$ 
W_{\tilde{b},i}:=\phi^{-1}_{\tilde{b},i}(Z_{\tilde{b}}) \setminus \Reg(X_{\tilde{b},i})~\mbox{is closed in}~\phi^{-1}_{\tilde{b},i}(Z_{\tilde{b}})
$$
by Lemma \ref{lemmaopenset}. Hence $W_{\tilde{b},i}$ is a closed subset of $X_{\tilde{b},i}$ and thus, $\phi_{\tilde{b},i}(W_{\tilde{b},i})$ is a constructible subset of $\Spec B[\tilde{b}^{-1}]$ by Chevalley's theorem (cf. \cite[Theorem 6]{Ma80}). 
From this, it follows that the Zariski closure $\overline{\phi_{\tilde{b},i}(W_{\tilde{b},i})}$ of $\phi_{\tilde{b},i}(W_{\tilde{b},i})$ is the set of all points of $\Spec B[\tilde{b}^{-1}]$ that are obtained as a specialization of a point of $\phi_{\tilde{b},i}(W_{\tilde{b},i})$.

Assume that we have $\fq \in \overline{\phi_{\tilde{b},i}(W_{\tilde{b},i})}$. Then there is a point $s \in W_{\tilde{b},i}$ such that $\fq \in \overline{\{\phi_{\tilde{b},i}(s)\}}$. By (\ref{reg}), the local ring $\mathcal{O}_{X_{\tilde{b},i},s}$ is regular. Hence $s \in \Reg(X_{\tilde{b},i})$, which is a contradiction to $s \in W_{\tilde{b},i}$. Thus, we must get $\fq \notin \overline{\phi_{\tilde{b},i}(W_{\tilde{b},i})}$. 

Let us choose an element $0 \ne b \in B$ such that $\fq \in \Spec B[b^{-1}] \subset \Spec B[\tilde{b}^{-1}]$ and $\Spec B[b^{-1}] \cap \overline{\phi_{\tilde{b},i}(W_{\tilde{b},i})}=\emptyset$ for all $i$. Consider the fiber square:
\[
\begin{array}{ccccc}
X_{b,i} & \stackrel{\phi_{b,i}}{\longrightarrow} & \Spec B[b^{-1}] &\\
\hphantom{\scriptstyle \psi_i}\downarrow{\scriptstyle} &&  \hphantom{\scriptstyle \psi_i}\downarrow{\scriptstyle}
& & \\
X_{\tilde{b},i} & \stackrel{\phi_{\tilde{b},i}}{\longrightarrow} & \Spec B[\tilde{b}^{-1}]\\
\end{array}
\]
Then we have $W_{\tilde{b},i} \cap X_{b,i}=\emptyset$. Let us put
$$
Z_b:=\Spec \biggl(\frac{B[b^{-1}]}{JB[b^{-1}]}\biggl)~\mbox{and}~W_{b,i}:=\phi^{-1}_{b,i}(Z_b) \setminus \Reg(X_{b,i}).
$$
Since $Z_b=Z_{\tilde{b}} \cap \Spec B[b^{-1}]$, we get
$$
W_{b,i}=\big(\phi^{-1}_{\tilde{b},i}(Z_{\tilde{b}}) \cap X_{b,i}\big) \setminus \Reg(X_{b,i})=\big(\phi^{-1}_{\tilde{b},i}(Z_{\tilde{b}}) \cap X_{b,i}\big) \setminus \Reg(X_{\tilde{b},i})=W_{\tilde{b},i} \cap X_{b,i}=\emptyset.
$$
This proves the assertion (\ref{openset}).
\end{proof}

\begin{lemma}
\label{specialization}
Let $\{ \phi_i :X_i \to Y \}_{i = 1, \ldots, m}$ be an alteration covering and let $y_1,\ldots,y_l$ be a sequence of points in $Y$ such that $y_{j+1} \in \overline{\{y_j\}}$ for $j=1,\ldots,l-1$, where $\overline{\{y_j\}}$ denotes the Zariski closure of $\{ y_j \}$ in $Y$. Then there exist $i$ and a sequence of points $x_1,\ldots,x_l$ in $X_i$ such that $\phi_i(x_j)=y_j$ for $j = 1, \ldots, l$ and $x_{j+1} \in \overline{\{x_j\}}$ for $j=1,\ldots,l-1$.
\end{lemma}

\begin{proof}
By assumption, for $i = 1,\ldots,m$, there is a commutative diagram:
\[
\begin{array}{ccccc}
V_i & {\longrightarrow} & V &\\
\hphantom{\scriptstyle \psi_i}\downarrow{\scriptstyle \psi_i} &&  \hphantom{\scriptstyle f}\downarrow{\scriptstyle f}
& & \\
X_i & \stackrel{\phi_i}{\longrightarrow} & Y\\
\end{array}
\]
where $V=\bigcup_{i=1}^m V_i$ is a Zariski open covering and $V \xrightarrow{f} Y$ is a proper surjective map. 
Let us find a sequence $v_1,\ldots,v_l$ in $V$ such that $v_j$ maps to $y_j$ for $j = 1, \ldots, l$ and $v_{j+1} \in \overline{\{v_j\}}$
for $j = 1, \ldots, l-1$. 
First, lift $y_1$ to a point $v_1 \in V$ via $f$. 
Suppose that a sequence $v_1,\ldots,v_t$ in $V$ has been found such that $f(v_j) = y_j$ for $j = 1, \ldots, t$ and $v_{j+1} \in \overline{\{v_j\}}$ for $j=1,\ldots,t-1$. 
So let us find $v_{t+1} \in V$ with the required condition. 
Since $f$ is a proper map, $f(\overline{\{v_t\}})$ is equal to $\overline{\{y_t\}}$. Hence $y_{t+1} \in f(\overline{\{v_t\}})$ and there is a lift $v_{t+1}$ of $y_{t+1}$ such that $v_{t+1} \in \overline{\{v_t\}}$. 
Here, we have $v_l \in V_i$ for some $i$. 
Since $V_i$ is closed under generalizations, $v_1,\ldots,v_l$ are contained in $V_i$. Then we see that the sequence $x_1:=\psi_i(v_1),\ldots,x_l:=\psi_i(v_l)$ in $X_i$ satisfies the required conditions.
\end{proof}

\begin{proof}[Proof of Theorem \ref{BrodmannRotthaus}]
In order to show that $A$ is a $J_2$-ring, it is enough to prove that 
the regular locus of any finite $A$-algebra is open (cf. \cite[Theorem73]{Ma80}). Let $B$ be a Noetherian domain that is finitely generated as an $A$-module. It suffices to prove that $\Reg(B)$ contains a non-empty open subset of $\Spec B$ by Nagata's topological criterion (cf. \cite[Theorem 24.4]{Ma86}). Since $A \to B$ is module-finite, we have $IB \subset \rad(B)$ and $B/IB$ is quasi-excellent. Since $B$ is a $G$-ring, $B_{\fq}$ is quasi-excellent for any $\fq \in \Spec B$
by \cite[Theorem 76]{Ma80}. By Lemma \ref{alterationcover}, there exists $b_{\fq} \in B \setminus \fq$ and an alteration covering $\{\phi_{b_{\fq},i}:X_{b_{\fq},i} \to \Spec B[b_{\fq}^{-1}]\}_i$ such that
\begin{equation}
\label{openset2}
\phi_{b_{\fq},i}^{-1}\biggl(\Spec \frac{B[b_{\fq}^{-1}]}{IB[b_{\fq}^{-1}]}\biggl) \subset 
\Reg(X_{b_{\fq},i})~\mbox{for all}~i.
\end{equation}
There exists a family of elements $b_1,\ldots,b_s \in \{ b_{\fq} \mid \fq \in \Spec B \}$ such that
\begin{equation}
\label{opencovering}
\Spec B=\Spec B[b_1^{-1}] \cup \cdots \cup \Spec B[b_s^{-1}].
\end{equation}
There exists a finite family of alteration coverings $\{\phi_{b_j,i}:X_{b_j,i} \to \Spec B[b_j^{-1}]\}_{i}$ with the same property as $(\ref{openset2})$ by letting $b_j=b_{\fq}$. 
By the lemma of generic flatness (cf. \cite[(22.A)]{Ma80}), we can find an element $0 \ne c \in B$ such that the induced map
\begin{equation*}
\phi_{b_j,i}^{-1}\big(\Spec B[b_j^{-1}c^{-1}]\big) \to \Spec B[b_j^{-1}c^{-1}]~\mbox{is flat for all}~i~\mbox{and}~j.
\end{equation*}

If we can show that $B[c^{-1}]$ is regular, then the proof is finished.

Let us pick $\fp \in \Spec B$ such that $c \notin \fp$. Then we want to prove that $B_{\fp}$ is regular. Let us choose a maximal ideal $\fm  \in \Spec B$ such that $\fp \subset \fm$. Since $IB \subset \rad(B)$, it follows that $IB \subset \fm$. 
By (\ref{opencovering}), there exists $j$ such that $\fp, \fm \in \Spec B[b_j^{-1}]$.
By Lemma~\ref{specialization}, there exist $x_1,x_2 \in X_{b_j,i}$ for some $i$ such that $x_2 \in \overline{\{x_1\}}$, $\phi_{b_j,i}(x_1)=\fp$ and $\phi_{b_j,i}(x_2)=\fm$. 
By (\ref{openset2}) together with the fact $IB \subset \fm$, the local ring $\mathcal{O}_{X_{b_j,i},x_2}$ is regular. Since $x_1$ is a generalization of $x_2$, $\mathcal{O}_{X_{b_j,i},x_1}$ is also a regular local ring. Here, $B_{\fp} \to \mathcal{O}_{X_{b_j,i},x_1}$ is flat, as $c \notin \fp$. Therefore, $B_{\fp}$ is regular by \cite[Proposition~17.3.3 (i)]{EGA20}.
\end{proof}

\section{Lifting problem for quasi-excellent rings}
\label{quasi-excellent4}

In this section, we shall prove the main theorem:

\begin{theorem}[Nishimura-Nishimura, Gabber]
\label{quasi-excellent}
Let $A$ be a Noetherian ring, and $I$ an ideal of $A$.
Assume that $A$ is $I$-adically complete.
Then, if $A/I$ is quasi-excellent, so is $A$.
\end{theorem}

\begin{proof}
Assume the contrary. Let $A$ be a Noetherian ring, and $I$ an ideal of $A$.
Suppose that $A$ is $I$-adically complete, $A/I$ is quasi-excellent, but $A$ is not quasi-excellent.

\vspace{2mm}

\noindent
{\bf Step~1.}
We shall reduce this problem to a simpler case as long as possible.

\vspace{2mm}

\noindent
{\bf (1-1)}
The following are well-known facts.
\begin{itemize}
\item
Let $R$ be a Noetherian ring and $I_{1}$, $I_{2}$ be ideals of $R$
such that $I_{1} \supset I_{2}$.
If $R$ is $I_{1}$-adically complete, then $R$ is $I_{2}$-adically complete.
\item
Let $R$ be a Noetherian ring and $J_{1}$, $J_{2}$ be ideals of $R$.
If $R$ is $J_{1}$-adically complete, then $R/J_{2}$ is $((J_{1}+J_{2})/J_{2})$-adically complete.
\end{itemize}
We shall use these facts without proving them.

Suppose $I = (a_{1}, \ldots, a_{t})$.
Put $I_i = (a_{1}, \ldots, a_{i})$ for $i =1,\ldots,t$ and $I_0 = (0)$.
Remark that $A/I_i$ is $(I_{i+1}/I_i)$-adically complete, and
\[
(A/I_i)/(I_{i+1}/I_{i}) = 
A/I_{i+1}
\]
for $i = 0, 1, \ldots, t-1$.
Here, $I_{i+1}/I_{i}$ is a principal ideal of $A/I_i$.
Remember that $A/I_t$ is quasi-excellent, but $A/I_0$ is not so.
Therefore, there exists $i$ such that $A/I_{i+1}$ is quasi-excellent, but $A/I_i$ is not so.

Replacing $A/I_i$ and $I_{i+1}/I_i$ with $A$ and $I$ respectively, we may assume that 
\begin{itemize}
\item[(A1)]
$I$ is a principal ideal generated by some $x \neq 0$, that is, $I=(x)$.
\end{itemize}

\vspace{2mm}

\noindent
{\bf (1-2)}
We put
\[
{\mathcal F} = \{J \mid \mbox{$A/J$ is not quasi-excellent} \} .
\]
Since ${\mathcal F}$ contains $(0)$, the set ${\mathcal F}$ is not empty and there exists a maximal element $J_{0}$ in ${\mathcal F}$. Replacing $A/J_{0}$ by $A$, we may assume that
\begin{itemize}
\item[(A2)]
if $J \neq (0)$, then $A/J$ is quasi-excellent.
\end{itemize}

\vspace{2mm}

\noindent
{\bf (1-3)}
By Theorem \ref{BrodmannRotthaus}, $A$ is not a $G$-ring.
There exist prime ideals $P$ and $Q$ of $A$ such that
$P \supset Q$, and the generic fiber of
\[
A_{P}/QA_{P} \longrightarrow (A_{P}/QA_{P})^\wedge
\]
is not geometrically regular.
On the other hand, if $Q \neq (0)$,  the above map is a regular homomorphism by (A2). Therefore we know $Q = (0)$, that is, $A$ is an integral domain.

Since quasi-excellent rings are Nagata \cite[Theorem 78]{Ma80}, we know that $A/xA$ is a Nagata ring.
Since the lifting property holds for Nagata rings by Marot \cite{Ma75},
$A$ is a Nagata domain. Let $\overline{A}$ be the integral closure of $A$ in $Q(A)$. Then $\overline{A}$ is module-finite over $A$.

By Greco's theorem \cite[Theorem 3.1]{Gre76}, $\overline{A}$ is not quasi-excellent, since $A$ is not so. Here $\overline{A}$ is $x \overline{A}$-adically complete and $\overline{A}$ satisfies (A2). Replacing $\overline{A}$ with $A$, we further assume that 
\begin{itemize}
\item[(A3)]
$A$ is a Nagata normal domain.
\end{itemize}

\vspace{2mm}

\noindent
{\bf (1-4)}
Since $A/xA$ is quasi-excellent, $A/xA$ is a Nagata $Z$-ring.
Since the lifting property holds for Nagata $Z$-rings by 
Nishimura-Nishimura \cite[Theorem A]{Nishimura87}, $A$ is also a Nagata $Z$-ring. Since $A$ is a normal $Z$-ring, $\widehat{A_{P}}$ is a local normal domain for any prime ideal $P$ of $A$. Then by \cite[Theorem 31.6]{Ma86}, $A$ is universally catenary. Thus we know
\begin{itemize}
\item[(A4)]
$A$ is a $Z$-ring and universally catenary.
\end{itemize}

\noindent
{\bf Step~2.}
We assume (A1), (A2), (A3) and  (A4).

By Theorem \ref{BrodmannRotthaus}, $A$ is not a $G$-ring.
By \cite[Theorem 75]{Ma80},
there exists a maximal ideal $\fm$ of $A$ such that 
the homomorphism $A_{\fm} \rightarrow \widehat{A_{\fm}}$
is not regular. Consider the fibers of $A_{\fm} \rightarrow \widehat{A_{\fm}}$. By (A2), the fibers except for the generic fiber are geometrically regular. So we concentrate on the generic fiber.

Let $L$ be a finite algebraic extension of $Q(A)$, 
where $Q(A)$ is the field of fractions of $A$.
Let $B_{L}$ be the integral closure of $A$ in $L$.
Remark that $B_L$ is a finite $A$-module by (A3).

Consider the following cofiber squares:
\[
\begin{array}{ccccccc}
L & = & L & \longrightarrow & L\otimes_{A}\widehat{A_{\fm}} & = & 
\prod_{\fn} L \otimes_{B_{L}} \widehat{(B_{L})_{\fn}} \\
\uparrow & & \uparrow & & \uparrow & & \\
B_{L} & \longrightarrow & B_{L} \otimes_{A}A_{\fm} & \longrightarrow &
B_{L} \otimes_{A} \widehat{A_{\fm}} & = & \prod_\fn \widehat{(B_{L})_{\fn}} \\
\uparrow & & \uparrow & & \uparrow & & \\
A & \longrightarrow & A_\fm  & \longrightarrow & \widehat{A_\fm} 
\end{array}
\]
Here $\fn$ runs over all the maximal ideals of $B_L$ lying over $\fm$.

We want to discuss whether $L \otimes_{B_{L}} \widehat{(B_{L})_{\fn}}$ is regular or not for $L$ and $\fn$.

Recall that $(B_L)_\fn/x (B_L)_\fn$ is a $G$-ring since $A/xA$ is quasi-excellent.
Let $(B_L)_\fn^*$ be the $x(B_L)_\fn$-adic completion of $(B_L)_\fn$.
By the local lifting property (Definition~\ref{GrothendieckLifting} in Appendix~\ref{appA}) for $G$-rings (Rotthaus \cite{Rott79}),
\begin{equation}\label{$G$-ring}
\mbox{$(B_L)_\fn^*$ is a $G$-ring.}
\end{equation}
Hence, the homomorphism $(B_L)_\fn^* \rightarrow \widehat{(B_L)_\fn^*} = \widehat{(B_L)_\fn}$
is regular and ${\rm Sing}\big((B_L)_\fn^*\big)$ is a closed subset of $\Spec (B_L)_\fn^*$. Therefore, $L \otimes_{B_L}(B_L)_\fn^*$ is regular if and only if
 $L \otimes_{B_L}\widehat{(B_L)_\fn}$ is regular.
For a finite algebraic extension $L$ of $Q(A)$, we put 
\[
\fS_L =
\left\{ \fq_{\fn}^* \ \ \left| \ \
\begin{array}{l}
\mbox{$\fq_{\fn}^*$ is a minimal prime ideal of  ${\rm Sing}\big((B_L)_\fn^*\big)$
such that $\fq_{\fn}^* \cap B_L = (0)$, } \\
\mbox{where $\fn$ is a maximal ideal  of $B_L$. } \\
\end{array}
\right.
\right\}
\]
Here, $L \otimes_{B_L}(B_L)_\fn^*$ is not regular for some $L$ and $\fn$, since $A$ is not a $G$-ring.
Therefore, $\fS_L$ is not empty for some $L$.
We put 
\begin{equation}\label{h_0}
h_0 = \min\{ \Ht Q \mid \mbox{$Q \in \fS_L$ for some $L$} \} .
\end{equation}

Remember that $A$ is a $Z$-ring by (A4).
Since $B_L$ is a finite $A$-module, $B_L$ is also a $Z$-ring. It is easy to see that $\widehat{(B_L)_\fn}$ is the completion of both 
$(B_L)_\fn$ and $(B_L)_\fn^*$.
Since $(B_L)_\fn$ and  $(B_L)_\fn^*$ are $Z$-rings (cf.\ (\ref{$G$-ring})),
\begin{equation}\label{normal}
\mbox{$(B_L)_\fn$, $(B_L)_\fn^*$ and $\widehat{(B_L)_\fn}$ are 
 local normal domains.}
\end{equation}
Therefore we know 
\begin{equation}\label{h_0 ge 2}
h_0 \ge 2 .
\end{equation}

We shall prove the following claim in the rest of Step 2.

\begin{claim}\label{ex}
Let $\fp$ be a prime ideal of $A$.
If $\Ht \fp \le h_0$, then $A_\fp$ is excellent.
\end{claim}

Let $\fp$ be a prime ideal of $A$ such that $0 < \Ht \fp \le h_0$.
By (A4), $A$ is universally catenary.
It is enough to prove that $A_\fp$ is a $G$-ring.
By (A2) and \cite[Theorem 75]{Ma80}, it is enough to show that the generic fiber of $A_\fp \rightarrow \widehat{A_\fp}$ is geometrically regular. Let $L$ be a finite algebraic extension of $Q(A)$, and $B_L$ be the integral closure of $A$ in $L$. Consider the following cofiber squares.
\[
\begin{array}{ccccccc}
L & = & L & \longrightarrow & L\otimes_{A}\widehat{A_{\fp}} & = & 
\prod_{\fq} L \otimes_{B_{L}} \widehat{(B_{L})_{\fq}} \\
\uparrow & & \uparrow & & \uparrow & & \\
B_{L} & \longrightarrow & B_{L} \otimes_{A}A_{\fp} & \longrightarrow &
B_{L} \otimes_{A} \widehat{A_{\fp}} & = & \prod_\fq \widehat{(B_{L})_{\fq}} \\
\uparrow & & \uparrow & & \uparrow & & \\
A & \longrightarrow & A_\fp  & \longrightarrow & \widehat{A_\fp} 
\end{array}
\]
Here $\fq$ runs over all the prime ideals of $B_L$ lying over $\fp$.
Since $A$ is normal, we have
\begin{equation}\label{pq}
0 < \Ht \fq = \Ht \fp \le h_0.
\end{equation}

It is enough to show that $L \otimes_{B_{L}} \widehat{(B_{L})_{\fq}}$
is regular. Let $\fn$ be a maximal ideal of $B_L$ such that $\fn \supset \fq$.
Let $\fq^*$ be a minimal prime ideal of $\fq (B_L)_\fn^*$.
Since $(B_L)_\fn \rightarrow (B_L)_\fn^*$ is flat, we have $\fq^* \cap B_L = \fq$. Consider the commutative diagram:
\[
\begin{array}{ccccccc}
 &  & (B_{L})_{\fn}^* & \longrightarrow & ((B_{L})_{\fn}^*)_{\fq^*} & 
\stackrel{\beta}{\longrightarrow} & ( ((B_{L})_{\fn}^*)_{\fq^*} )^\wedge \\
& & \uparrow & & \hphantom{\scriptstyle \alpha}\uparrow{\scriptstyle \alpha} & &
 \hphantom{\scriptstyle \hat{\alpha}}\uparrow{\scriptstyle \hat{\alpha}}  \\
B_{L} & \longrightarrow & (B_{L})_\fn
& \longrightarrow &
(B_{L})_\fq  & \longrightarrow  & \widehat{(B_{L})_\fq} 
\end{array}
\]
The map $\alpha$ as above is a flat local homomorphism.
Since the closed fiber of $\alpha$ is of dimension $0$, we have
$\dim (B_{L})_\fq = \dim ((B_{L})_{\fn}^*)_{\fq^*}$.
In particular 
\begin{equation}\label{q*}
\Ht \fq = \Ht \fq^*.
\end{equation}

By (\ref{$G$-ring}), $\beta$ is a regular homomorphism
and ${\rm Sing}\big((B_{L})_{\fn}^*\big)$ is a closed subset of
$\Spec (B_{L})_{\fn}^*$.

Let $c_\fn^*$ be the defining ideal of ${\rm Sing}\big((B_{L})_{\fn}^*\big)$
satisfying $c_\fn^* = \sqrt{c_\fn^*}$.
We put $T = ( ((B_{L})_{\fn}^*)_{\fq^*} )^\wedge$.
Since $\beta$ is a regular homomorphism, $c_\fn^*T$ is a defining ideal of ${\rm Sing}(T)$.
Suppose
\[
c_\fn^* = \fq_{\fn,1}^* \cap \cdots \cap \fq_{\fn,s}^* ,
\]
where $\fq_{\fn, i}^*$'s are prime ideals of $(B_{L})_{\fn}^*$
such that $\fq_{\fn, i}^* \not\supset \fq_{\fn, j}^*$ if $i \neq j$.

One of the following three cases occurs:
\begin{itemize}
\item[Case 1.]
If none of $\fq_{\fn, i}^*$'s is contained in $\fq^*$, then
$c_\fn^*T = T$.
\item[Case 2.]
If $\fq_{\fn, i}^* = \fq^*$ for some $i$, then
$c_\fn^*T = \fq^*T$.
\item[Case 3.]
Suppose that $\fq_{\fn, 1}^*$, \ldots,  $\fq_{\fn, t}^*$ are properly contained in $\fq^*$, and  $\fq_{\fn, t+1}^*$, \ldots,  $\fq_{\fn, s}^*$ are not contained in $\fq^*$
for some $t$ satisfying $1 \le t \le s$.
Then, $c_\fn^*T = \fq_{\fn, 1}^*T \cap \cdots \cap \fq_{\fn, t}^*T$.
\end{itemize}

In any case as above, we can verify $c_\fn^*T \cap B_L \neq (0)$ as follows. In Case 1, we have $c_\fn^*T \cap B_L= B_L \ne (0)$. In Case 2, we have $\fq^*T \cap B_L = \big(\fq^*T \cap (B_L)_\fn^*\big) \cap B_L = \fq^* \cap B_L = \fq \neq (0)$ by (\ref{pq}). In Case 3, suppose that $\fq_{\fn, i}^*$ is properly contained in $\fq^*$. Then by (\ref{pq}) and (\ref{q*}), we have $\Ht \fq_{\fn, i}^* < h_0$. Since $\fq_{\fn, i}^*$ is in ${\rm Sing}\big((B_L)_\fn^*\big)$,  we have $\fq_{\fn, i}^* \cap B_L \neq (0)$ by the minimality of $h_0$.

Take $0 \neq b \in c_\fn^*T \cap B_L$.
Since  $c_\fn^*T$ is the defining ideal of ${\rm Sing}(T)$,
$T \otimes_{B_L} B_L[b^{-1}]$ is regular.
Since 
\[
 \widehat{(B_{L})_\fq} \otimes_{B_L} B_L[b^{-1}] \stackrel{\hat{\alpha}\otimes 1}{\longrightarrow}
T \otimes_{B_L} B_L[b^{-1}]
\]
is faithfully flat,
$\widehat{(B_{L})_\fq} \otimes_{B_L} B_L[b^{-1}]$ is regular.
Hence, $\widehat{(B_{L})_\fq} \otimes_{B_L} L$ is regular.
We have completed the proof of Claim \ref{ex}.

\vspace{2mm}

\noindent
{\bf Step 3.}
Here, we shall complete the proof of Theorem \ref{quasi-excellent}.

First of all, remember  the following Rotthaus' Hilfssatz (cf. \cite[Theorem 1.9 and Proposition 1.18]{Nishimura87} or originally \cite{Rott80}).

\begin{theorem}[Rotthaus' Hilfssatz]
\label{Rotthaus' Hilfssatz}
Let $B$ be a Noetherian ring and $\fn \in \Max B$.
Assume that $B$ is $xB$-adically complete Nagata ring.

We put
\[
\Gamma(\fn) = \{ \gamma  \mid \fn \in \gamma \subset \Max B, \ {}^\#\gamma < \infty \} .
\]
For $\gamma \in \Gamma(\fn)$, we put $S_\gamma = B \setminus \cup_{\fa \in \gamma}\fa$ and $B_\gamma = S_\gamma^{-1}B$. Consider the homomorphism $B_\gamma^* \rightarrow B_\fn^*$ induced by $B_\gamma \rightarrow B_\fn$, where $( \ )^*$ denotes the $(x)$-adic completion.

Let $\fq_\fn^*$ be a minimal prime ideal of $\Sing(B_\fn^*)$.
For each $\gamma \in \Gamma(\fn)$, we put 
$\fq_\gamma^*:=\fq_\fn^* \cap B_\gamma^*$. We define
\[
\Delta_\gamma(x):=\left\{ 
\overline{Q} \cap B \mid
\overline{Q} \in {\rm Min}_{\overline{(B_\gamma^*/\fq_\gamma^*)}}
\big(\overline{(B_\gamma^*/\fq_\gamma^*)}/(x)\big)
\right\} , \ \ \Delta(x):=\bigcup_{\gamma \in \Gamma(\fn)} \Delta_\gamma(x)
\]
where $\overline{(B_\gamma^*/\fq_\gamma^*)}$ is the normalization of $B_\gamma^*/\fq_\gamma^*$ in the field of fractions.

Assume the following two conditions:
\begin{itemize}
\item[(i)]
for each $\gamma \in \Gamma(\fn)$, $\Ht \fq_\gamma^* > 0$,
\item[(ii)]
${}^\# \triangle (x) < \infty$.
\end{itemize}

Then $\fq_\fn^* \cap B \neq (0)$ is satisfied.
\end{theorem}

We refer the reader to \cite{Nishimura87} for the proof of this theorem. We deeply use it in our proof.

\vspace{1mm}

Now, we start to prove Theorem \ref{quasi-excellent}.

Let $L$ be a finite algebraic extension of $Q(A)$. Let $B$ be the integral closure of $A$ in $L$. Let $\fn$ be a maximal ideal of $B$. Suppose that $\fq_{\fn}^*$ is a minimal prime ideal of ${\rm Sing}(B_\fn^*)$ such that 
\begin{equation}\label{intersection}
\fq_{\fn}^* \cap B=(0)
\end{equation}
 and $\Ht \fq_{\fn}^*=h_0$.
We remark that such $L$, $B$, $\fn$, $\fq_\fn^*$ certainly exist by the definition of $h_0$ (see (\ref{h_0})).

We define $B_\fn^*$, $B_\gamma$, $B_\gamma^*$, $\fq_\gamma^*$, $\Delta_\gamma(x)$, $\Delta(x)$ as in Theorem~\ref{Rotthaus' Hilfssatz}.
Recall that $B_\gamma$ is a semi-local ring satisfying 
${\rm Max}(B_\gamma) = \{ \fa B_\gamma \mid \fa \in \gamma \}$.
Since $B$ is $xB$-adically complete, any maximal ideal of $B$ contains $x$.
Since $B_\gamma/xB_\gamma$ is isomorphic to $B_\gamma^*/xB_\gamma^*$, 
there exists one-to-one correspondence between $\gamma$
and the set of maximal ideals of $B_\gamma^*$.
Let $\fn^*$ be the maximal ideal of $B_\gamma^*$
corresponding to $\fn$, that is, $\fn^* = \fn B_\gamma^*$.

\vspace{2mm}

In the rest of this proof, we shall prove the conditions (i) and (ii) in Theorem \ref{Rotthaus' Hilfssatz}. Then it contradicts (\ref{intersection}) and this completes the proof of Theorem \ref{quasi-excellent}.

\vspace{2mm}

Put $C_\gamma = B_\gamma^*/\fq_\gamma^*$.
Since $A$ is a Nagata ring, $C_\gamma$ is a Nagata ring, too.
Therefore, the normalization $\overline{C_\gamma}$ is a
finite $C_\gamma$-module.
Note that $x \not\in \fq_\gamma^*$ since $\fq_\fn^* \cap B = (0)$.
Take
$\overline{Q} \in 
{\rm Min}_{\overline{C_\gamma}}(\overline{C_\gamma}/x \overline{C_\gamma})$.
Since $C_\gamma$ is universally catenary,
the dimension formula holds between $C_\gamma$ and $\overline{C_\gamma}$
(\cite[Theorem 15.6]{Ma86}).
Thus we know that $\overline{Q} \cap C_\gamma$ is a minimal prime ideal of $xC_\gamma$.
Therefore, $\triangle_\gamma (x)$ defined in Theorem \ref{Rotthaus' Hilfssatz}
coincides with
\[
\{ \tilde{Q} \cap B \mid \tilde{Q} \in {\rm Min}_{B_\gamma^*}\left( B_\gamma^*/
\fq_\gamma^* + xB_\gamma^* \right) \} .
\]

Here, we shall prove the following claim.

\begin{claim}
\label{ht}
\begin{enumerate}
\item
For each $\gamma \in \Gamma(\fn)$, 
$\Ht \fq_\gamma^* = h_0$.
\item
For any $Q \in \triangle (x)$, $\Ht Q = h_0 + 1$.
\end{enumerate}
\end{claim}

First, we shall prove (1). Consider the following homomorphisms.
\begin{equation}\label{gf}
(B_\gamma^*)_{\fn^*} \stackrel{f}{\longrightarrow}
B_\fn^* \stackrel{g}{\longrightarrow} \widehat{B_\fn},
\end{equation}
where $\fn^*=\fn B_{\gamma}^*$. By the local lifting property for $G$-rings, $B_\gamma^*$ is a $G$-ring.
Since $\widehat{B_\fn} = \widehat{(B_\gamma^*)_{\fn^*}}$,
$gf$ is a regular homomorphism.
Since $g$ is faithfully flat, $f$ is a regular homomorphism
by \cite[(33.B)]{Ma80} or \cite[Theorem 32.1]{Ma86}.
Since $\fq_\gamma^*(B_\gamma^*)_{\fn^*} = \fq_\fn^* \cap (B_\gamma^*)_{\fn^*}$,
\begin{equation}\label{min}
\mbox{$\fq_\gamma^*(B_\gamma^*)_{\fn^*}$ is a minimal prime ideal of
${\rm Sing}((B_\gamma^*)_{\fn^*})$.}
\end{equation}
Furthermore, $\fq_\fn^*$ is a minimal prime ideal of $\fq_\gamma^*B_\fn^*$.
Then, we have
\[
\Ht \fq_\gamma^* = \Ht \fq_\gamma^* (B_\gamma^*)_{\fn^*}
 = \Ht \fq_\fn^* = h_0.
\]
The assertion (1) has thus been proved. Since $h_0 \ge 2$ as in (\ref{h_0 ge 2}), the condition (i) in Theorem \ref{Rotthaus' Hilfssatz} follows from the above assertion (1).

Next, we prove (2).
Take $Q \in \triangle_\gamma(x)$ for some $\gamma \in \Gamma(\fn)$.

We shall prove that $B_\gamma^*$ is normal. For $\fa \in \gamma$, let $\fa^*$ denote the maximal ideal $\fa B_\gamma^*$ of $B_\gamma^*$. Here $B$ is a normal $Z$-ring, since $A$ is.
Therefore, $\widehat{B_\fa}$ is a normal local domain.
The map $(B_\gamma^*)_{\fa^*} \rightarrow \widehat{(B_\gamma^*)_{\fa^*}}
= \widehat{B_\fa}$ is a faithfully flat morphism.
Then $(B_\gamma^*)_{\fa^*}$ is normal by the corollary of \cite[Theorem 23.9]{Ma86}.
Hence, we know that $B_\gamma^*$ is normal.

By definition, there exists $\tilde{Q} \in {\rm Min}_{B_\gamma^*}\left( B_\gamma^*/
\fq_\gamma^* + xB_\gamma^* \right)$ such that $Q = \tilde{Q} \cap B$.
Since $B_\gamma \rightarrow B_\gamma^*$ is flat and 
$QB_\gamma^* = \tilde{Q}$, we have
\begin{equation}\label{Ht1}
\Ht \tilde{Q} = \dim (B_\gamma^*)_{\tilde{Q}} = \dim (B_\gamma)_Q = \Ht Q .
\end{equation}
Since $B_\gamma^*$ is a Noetherian normal ring, $B_\gamma^*$ is the direct product of finitely many integrally closed domains.
Remember that $B_\gamma^*$ is universally catenary.
Then we have
\begin{equation}\label{Ht2}
\Ht  \tilde{Q} = \Ht \fq_\gamma^* + 1 .
\end{equation}
By (\ref{Ht1}), (\ref{Ht2}) and the assertion (1) together, we obtain $\Ht Q = h_0+1$.

We have completed the proof of Claim \ref{ht}.

Let $\fq$ be a prime ideal of $B$ such that $\Ht \fq \le h_0$.
Since $A$ is normal, $\Ht (\fq \cap A) = \Ht \fq \le h_0$.
Hence, $A_{(\fq \cap A)}$ is excellent by Claim \ref{ex}.
Therefore, $B_\fq$ is excellent. Let us remember that $B/xB$ is quasi-excellent. Then by Lemma \ref{alterationcover},
there exists $b_\fq \in B \setminus \fq$ such that there exists an alteration covering
\begin{equation}\label{covering}
\{ X_{b_\fq, i} \stackrel{\phi_{b_\fq, i}}{\longrightarrow} 
{\rm Spec}(B[b_\fq^{-1}]) \}_i
\end{equation}
such that
\begin{equation}\label{regular}
\phi_{b_\fq, i}^{-1}\biggl(\Spec \dfrac{B[b_\fq^{-1}]}{xB[b_\fq^{-1}]}\biggl)
\subset {\rm Reg}(X_{b_\fq, i})
\end{equation}
for each $i$.
We put
\[
\Omega = \bigcup_{\fq} \Spec(B[b_\fq^{-1}])
\subset \Spec(B) .
\]
By definition, $\Omega$ is an open set that contains all the prime ideals of $B$ 
of height less than or equal to $h_0$.
Hence, the complement $\Omega^c$ contains only finitely many prime ideals
of height $h_0 + 1$.
If $\triangle (x)$ is contained in $\Omega^c$, then
$\triangle (x)$ must be a finite set by Claim~\ref{ht} (2).

Thus, it suffices to prove
\[
\triangle (x) \subset \Omega^c .
\]
Assume the contrary.
Take $Q \in \triangle (x) \cap \Omega$.
Since  $Q \in \triangle (x)$, there exists $\tilde{Q} \in  {\rm Min}_{B_\gamma^*}\left( B_\gamma^*/
\fq_\gamma^* + xB_\gamma^* \right)$ such that $\tilde{Q} \cap B = Q$
for some $\gamma \in \Gamma(\fn)$. Since  $Q \in \Omega$, we find that $Q \in \Spec(B[b_\fq^{-1}])$ for some $\fq$ with $\Ht \fq \le h_0$.

Since (\ref{covering}) is an alteration covering, we have a proper surjective generically finite map
$\pi:V \rightarrow \Spec(B[b_\fq^{-1}])$, together with an open covering $V = \bigcup_i V_i$ and a morphism $\psi_i:V_i \rightarrow X_{b_\fq, i}$ for each $i$ with 
commutative diagrams as in (\ref{diagram}).
Consider the following diagram:
\[
\begin{array}{ccccc}
V'_{i} & \subset & V' & & \\
\hphantom{\scriptstyle \psi'_i}\downarrow{\scriptstyle \psi'_i} & & \hphantom{\scriptstyle f'}\downarrow{\scriptstyle f'}
& & \\
X'_{b_\fq, i} & \stackrel{\phi'_{b_\fq,i}}{\longrightarrow} & \Spec(B_{\gamma}^{*}[b_\fq^{-1}])
 & \longrightarrow & \Spec(B_{\gamma}^{*}) \\
\hphantom{\scriptstyle g}\downarrow{\scriptstyle g} & & \hphantom{\scriptstyle h}\downarrow{\scriptstyle h} & & \downarrow \\
X_{b_\fq,i} & \stackrel{\phi_{b_\fq,i}}{\longrightarrow} & \Spec(B[b_\fq^{-1}])
 & \longrightarrow & \Spec(B)
\end{array}
\]
We put $( \ )'=( \ ) \times_{\Spec B} \Spec B_{\gamma}^{*}$.
Then, both $\tilde{Q}$ and $\fq_{\gamma}^{*}$ are contained in $\Spec(B_{\gamma}^{*}[b_\fq^{-1}])$.
Since $f':V' \rightarrow \Spec(B_{\gamma}^{*}[b_\fq^{-1}])$ is proper surjective, there exist $\xi_{1}, \xi_{2} \in V'$ such that $f'(\xi_{1})=\tilde{Q}$, $f'(\xi_{2})=\fq_{\gamma}^{*}$,
and $\xi_{1}$ is a specialization of $\xi_{2}$.
Since $V'=\bigcup_iV'_i$ is a Zariski open covering, we have $\xi_{1} \in V'_{i}$ for some $i$.
Since $V'_{i}$ is closed under generalization,
both $\xi_{1}$ and $\xi_{2}$ are contained in $V'_{i}$.
Here, we put $\eta_{1} = \psi'_i(\xi_{1})$, $\eta_{2} = \psi'_i(\xi_{2})$ and
$\zeta_{1} = g(\psi'_i(\xi_{1}))$.

Since $x \in Q$ and $\phi_{b_\fq, i}(\zeta_{1})=Q$,
we know that ${\mathcal O}_{X_{b_\fq, i}, \zeta_{1}}$ is a regular local ring by (\ref{regular}).
Since $B \rightarrow B_{\gamma}^{*}$ is flat, ${\mathcal O}_{X_{b_\fq, i}, \zeta_{1}} \rightarrow 
{\mathcal O}_{X'_{b_\fq, i}, \eta_{1}}$ is a flat local homomorphism.
Its closed fiber is the identity
since the maximal ideal of ${\mathcal O}_{X_{b_\fq, i}, \zeta_{1}}$ contains $x$.
Thus, ${\mathcal O}_{X'_{b_\fq, i}, \eta_{1}}$ is a regular local ring.
Since $\eta_{2}$ is a generalization of $\eta_{1}$, 
${\mathcal O}_{X'_{b_\fq, i}, \eta_{2}}$ is also a regular local ring.
Here, $(B_{\gamma}^*)_{\fq_{\gamma}^{*}} \rightarrow {\mathcal O}_{X'_{b_\fq, i}, \eta_{2}}$
is flat, since $h(\fq_{\gamma}^{*}) = (0)$ by (\ref{intersection}). Therefore, $(B_{\gamma}^*)_{\fq_{\gamma}^{*}}$ is a regular local ring. It contradicts (\ref{min}). The condition (ii) in Theorem \ref{Rotthaus' Hilfssatz} has been proved.

We have completed the proof of Theorem~\ref{quasi-excellent}.
\end{proof}

\appendix
\section{Lifting property, local lifting property, etc.}
\label{appA}

We introduce notions of the lifting property, the local lifting property, etc.
We make a table on the known results here.

\begin{definition}
\label{GrothendieckLifting}
Let $\mathbf{P}$ be a ring theoretic property of Noetherian rings.

\begin{enumerate}
\item
We say that the \textit{Lifting Property} ($LP$ for short) holds for $\mathbf{P}$, if the following condition holds. For any Noetherian ring $A$ with an ideal $I$ such that $A$ is $I$-adically complete, if $A/I$ is $\mathbf{P}$, then $A$ is $\mathbf{P}$.

\item
We say that the \textit{Local Lifting Property} ($LLP$ for short) holds for $\mathbf{P}$, if the lifting property holds for $\mathbf{P}$ in the case where the given ring $A$ is semi-local.

\item
We say that the  \textit{Power Series Extension Property} ($PSEP$ for short) holds for $\mathbf{P}$, if the following condition holds. For any Noetherian ring $A$, if $A$ is $\mathbf{P}$, then the formal power series ring $A[[x]]$ is $\mathbf{P}$.

\item
We say that the  \textit{Ideal-adic Completion Property} ($ICP$ for short) holds for $\mathbf{P}$, if the following condition holds. For any Noetherian ring $A$ with an ideal $I$, if $A$ is $\mathbf{P}$, then the $I$-adic completion of $A$ is $\mathbf{P}$.
\end{enumerate}
\end{definition}

For any property $\mathbf{P}$, it is easy to see the following implications:
\[
LLP \Longleftarrow LP \Longrightarrow PSEP .   
\]

If the category of Noetherian rings having the property $\mathbf{P}$ is closed
under homomorphic images, then the implication 
\[
PSEP \Longrightarrow ICP
\]
holds since $A[[x_1. \ldots, x_n]]/(x_1-a_1, \ldots, x_n-a_n)$ coincides with the $(a_1, \ldots, a_n)$-adic completion of $A$.

If the category of Noetherian rings having the property $\mathbf{P}$ is closed
under polynomial extension, then the implication 
\[
PSEP \Longleftarrow ICP
\]
holds since $A[[x_1. \ldots, x_n]]$ is the $(x_1, \ldots, x_n)$-adic completion of $A[x_1, \ldots, x_n]$.

In this article, we consider $\mathbf{P}$ as one of the following properties:
\[
\mbox{$G$-ring, quasi-excellent, excellent, universally catenary, Nagata, Nagata $Z$-ring}.
\]
Let us remark that the category of Noetherian rings with $\mathbf{P}$ as above is closed under polynomial extensions and homomorphic images~\footnote{Polynomial rings over a $G$-ring are also $G$-rings (cf.\ \cite[Theorem 77]{Ma80}). One can prove the same for $Z$-rings in the same way. 
The category of Nagata rings is closed under polynomial extensions, homomorphic images and localizations~(cf.\ \cite[Theorem 72]{Ma80}).}.

By Nagata-Ratliff theorem (\cite[Theorem 3.1]{Rat}),  PSEP holds for $\mathbf{P}=$ universally catenary. In 1975, Marot \cite{Ma75}  proved that LP holds for $\mathbf{P}=$ Nagata. In 1979, Rotthaus \cite{Rott79} proved that LLP holds for $\mathbf{P}=$ $G$-ring. Hence, LLP holds for $\mathbf{P}=$ quasi-excellent. In 1981, Nishimura \cite{Nishimura81} found an example, and proved that ICP does not hold for $\mathbf{P}=$ $G$-ring. In 1982, Greco \cite{Gre82} found an example, and proved that LLP does not hold for both $\mathbf{P}=$ universally catenary and  $\mathbf{P}=$ excellent. In 1987, Nishimura-Nishimura (\cite[Theorem A]{Nishimura87}) proved that LP holds for $\mathbf{P}=$ Nagata $Z$-ring. They also proved that  LP holds for $\mathbf{P}=$ quasi-excellent, if the ring contains a field of characteristic $0$ (\cite[Theorem B]{Nishimura87}). 
As in Main~Theorem~\ref{MT1}, Gabber proved that LP holds for $\mathbf{P}=$ quasi-excellent in general.
\[
\begin{tabular}{|c||c|c|c|c|}
\hline
& LLP & LP & PSEP & ICP \\ 
\hline\hline
\mbox{$G$-ring} & \hphantom{AA}$\bigcirc$\hphantom{AA} & \hphantom{AA}$\times$\hphantom{AA} & \hphantom{AA}$\times$\hphantom{AA} & \hphantom{AA}$\times$\hphantom{AA} \\ \hline
\mbox{quasi-excellent} & $\bigcirc$ & $\bigcirc$ & $\bigcirc$ & $\bigcirc$ \\ \hline
\mbox{excellent} &  $\times$ &  $\times$ & $\bigcirc$ & $\bigcirc$ \\ \hline
\mbox{universally catenary} &  $\times$ &  $\times$ & $\bigcirc$ & $\bigcirc$ \\ \hline
\mbox{Nagata} &  $\bigcirc$ &  $\bigcirc$ & $\bigcirc$ & $\bigcirc$ \\ \hline
\mbox{Nagata $Z$-ring} &   $\bigcirc$ &  $\bigcirc$ & $\bigcirc$ & $\bigcirc$ \\ \hline
\end{tabular}
\]
Remark that PSEP and ICP hold for $\mathbf{P}=$ excellent since they hold for both $\mathbf{P}=$ quasi-excellent and  $\mathbf{P}=$ universally catenary.
Thus Main~Theorem~\ref{MT2} follows.

\begin{acknowledgement}
The authors are grateful to Professor O. Gabber for permitting us to write this paper. The authors are also grateful to Professor J. Nishimura for listening to the proof of the main theorem and providing us with useful suggestions kindly. 
We also thank the referees for valuable comments.
The first author was partially supported by the Ministry of Education, Science, Sports and Culture, Grant-in-Aid for Scientific Research (C), 2015-2018（15K04828). The second author was partially supported by the Ministry of Education, Science, Sports and Culture, Grant-in-Aid for Scientific Research (C), 2018-2021（18K03257).
\end{acknowledgement}

\end{document}